\documentclass[11pt,a4paper]{article}
\usepackage[colorlinks]{hyperref}
\usepackage[margin=1in]{geometry}
\usepackage{authblk}
\usepackage{times,amsmath,stackengine,amsthm,amsfonts,amssymb,mathtools,color,latexsym,amsxtra,layout}
\usepackage{graphicx,subfig}
\usepackage{algorithmicx, algpseudocode}
\usepackage{algorithm}
\usepackage[font={small}]{caption}
\usepackage{cleveref}
\usepackage{ulem}
\usepackage{verbatim}
\usepackage{wrapfig,lipsum,booktabs}
\usepackage[mathscr]{euscript}
\usepackage{multirow}
\usepackage{bbm}
\usepackage[titletoc,title]{appendix}
\usepackage{mathabx}
\usepackage{spreadtab}
\usepackage{diagbox}
\usepackage{enumerate}

\newtheorem{corollary}{Corollary}

\newtheorem{lemma}{Lemma}

\usepackage{float}

\renewcommand{\epsilon}{\varepsilon}

\newcommand{\ds}{\displaystyle}

\newtheorem{thm}{Theorem}

\author[1]{Toni Sayah}
\author[2]{Toufic El Arwadi}

\affil[1]{Laboratoire de Math\'ematiques et Applications, Unit\'e
de recherche Math\'ematiques et
Mod\'elisation, CAR, Facult\'e des Sciences, Universit\'e Saint-Joseph de Beyrouth, B.P 11-514 Riad El Solh, Beyrouth 1107 2050, Lebanon\\

(toni.sayah@usj.edu.lb)}

\affil[2]{Department of Mathematics and Computer Science, Faculty of Science, Beirut Arab University, P.O. Box 11-5020, Beirut, Lebanon\\

(t.elarwadi@bau.edu.lb)}

\title{Exponential decay of the discrete energy for the wave-wave coupled system}

\date{}

\begin{document}

\maketitle
\begin{abstract}
\noindent In this article,  a numerical analysis of the asymptotic
behavior of the discrete energy associated to a dissipative
coupled wave system is conducted. The numerical approximation of
the system is constructed using the P1 finite element method for
spatial discretization, combined with the implicit Euler scheme
for time integration.  An a priori error analysis is established,
showing that, under extra regularity assumptions on the continuous
solution, the numerical scheme exhibits linear convergence. Then,
for the first time in the literature,  the exponential decay of
the fully discrete energy is shown using the energy method.
\end{abstract}
\section{Introduction}
The dissipation of energy for the damped wave equation has
received a lot of attention since the pioneering results of
Komornik \cite{Lagnese} and Lagnese \cite{Komornik}. More
specifically, let $\Omega$ be a bounded regular domain and
$u:\Omega\times (0,+\infty)\rightarrow \mathbb{R}$ the unique
solution of the following wave equation
\begin{equation} \label{onde}
\left\{
\begin{array}{lcl}
u_{tt}-  \Delta u +\alpha u_t =0 && \mbox{ on } \Omega \times (0,\infty), \\

u(x,t)=0 && \mbox{ on } \Gamma \times [0,\infty]
\end{array}
\right.
\end{equation}
The corresponding energy is $\displaystyle
E(t)=\frac{1}{2}\int_\Omega u_t^2 dx+\frac{1}{2}\int_\Omega
|\nabla u|^2dx$. Clearly $\ds \frac{d}{dt} E(t)=-\int_\Omega
\alpha u_t^2<0$. This shows that the energy is non-increasing and
approaches zero. Moreover, when $\alpha$ is constant or a strictly
bounded positive function, it is known that there exist two
positive constants  $M$ and $\kappa$ such that
\[E(t)\leq Me^{-\kappa t}. \]
What this model describes can easily be extended to the coupled
damped wave system which consists of two wave equations weakly
coupled through a frictional interaction. The coupled system takes
the form:
\begin{equation} \label{PDE}
\left\{
\begin{array}{lcl}
u_{tt}- c^2 \Delta u + \varepsilon_u u_t + \alpha u - \alpha v =0 && \mbox{ on } \Omega \times (0,T), \\
v_{tt}- c^2 \Delta v + \varepsilon_v v_t + \alpha v - \alpha u =0 && \mbox{ on } \Omega \times (0,T), \\
u(x,0) = u_0(x), \; u_t(x,0) = u_1(x) && \mbox{ on } \Omega, \\
v(x,0) = v_0(x), \; v_t(x,0) = v_1(x) && \mbox{ on } \Omega,
\\
u(x,t)=0 && \mbox{ on } \Gamma \times [0,T],\\
v(x,t)=0 && \mbox{ on } \Gamma \times [0,T],
\end{array}
\right.
\end{equation}
where $\epsilon_u \geq 0$ and $\epsilon_v\geq 0$ are the damping
coefficients, $\alpha > 0$ is the coupling coefficient,  $u_0.
u_1, v_0$ and $v_1$ are given functions.

\noindent The general stabilization properties of the system have
attracted considerable attention. Specifically, Alabau et al.
\cite{Alabau-1} demonstrated that for $\varepsilon_u = 0$, the
energy of the system decays over time with a polynomial decay
rate. This result was later improved by Lobato et al.
\cite{Lobato}. Furthermore, the asymptotic behavior of coupled
damped wave systems is a subject of considerable literature for
the different configurations of damping (see, for instance,
\cite{KV}).

\noindent Regarding the numerical approach, Lobato et al.
\cite{Lobato} proposed for the first time a fully discrete scheme
for the one-dimensional coupled wave system based on a finite
difference in space and time. In addition, they demonstrate that
the discrete energy satisfies a monotonicity of the form $E^n <
E^{n-1}$. Other authors studied energy decay in dissipative
systems by applying time-implied Euler schemes combined with space
finite element or finite difference discretizations (see, e.g.,
\cite{fernandez, Copetti, Copetti2, mauro, fernandez-2, Aouadi2,
zuazua, Sayah1, Sayah2}).

\noindent To our knowledge, the nature of the exponential decay of
the discrete energy for the coupled damped wave system \eqref{PDE}
has not been studied.
This is the goal we pursue: to develop a numerical scheme based on finite elements and prove the energy decay for the discrete energy is indeed exponential. \\

\noindent In Section 2, we present the system and we show the
exponential decay of the continuous energy, in Section 3 a
numerical scheme is proposed and its well-posdness is proved. The
aim of section 3 is to show an {\it a priori} error estimate. In
section 4, we present the discrete energy and show its exponential
decay in times $t_n$.

\section{exponential decay of the energy}
In this section, we introduce an energy corresponding to the problem \eqref{PDE} and show the corresponding exponential decay.\\

\noindent We define the following energy:
\[
E(t)= \displaystyle \frac{1}{2} \int_\Omega \big( |u_t|^2 +
|v_t|^2 + c^2 |\nabla u |^2 + c^2|\nabla v |^2 + \alpha (u - v)^2
\big).
\]
\begin{lemma}
The positive energy $E$ decays with respect to time.
\end{lemma}
\begin{proof}
W begin by computing the derivative of the energy $E(t)$:
\begin{equation}\label{energyp}
\begin{array}{rcl}
\quad
E'(t) &=& (u_{tt},u_t) + (v_{tt},v_t) + c^2(\nabla u, \nabla u_t) + c^2(\nabla v, \nabla v_t) \\
\quad
&& + \alpha(u,u_t) + \alpha (v,v_t) - \alpha(v,u_t) - \alpha(u,v_t)\\
&=& - \displaystyle \varepsilon_u ||u_t ||^2 - \displaystyle
\varepsilon_v ||v_t ||^2.
\end{array}
\end{equation}
We deduce the decay of the energy.
\end{proof}
\begin{lemma}
The energy $E$ decays exponentially with time: there exist
positive constants $C$ and $\eta$ such that
\[
E(t) \le C e^{-\eta t} E(0).
\]
\end{lemma}
\begin{proof}
We proceed by several steps:
\begin{enumerate}
\item We introduce the following function:
\[
L(t) = N E(t) + f(t),
\]
where $N$ is a large real number and $f(t) = \beta (u_t, u) +
\beta (v_t, v)$ for a positive real number $\beta$. Then, we have
\begin{equation}\label{Lprim}
L'(t) = N E'(t) + f'(t),
\end{equation}
where
\begin{equation}\label{Epr}
E'(t)= - \displaystyle \varepsilon_u ||u_t ||^2 - \displaystyle
\varepsilon_v ||v_t ||^2
\end{equation}
and
\begin{equation}\label{frp}
\begin{array}{rcl}
\medskip
f'(t) &=& \beta (u_{tt},u) + \beta (v_{tt},v) + \beta (u_t, u_t) + \beta (v_t, v_t)\\
&=& -\beta c^2 ||\nabla u ||^2 - \varepsilon_u \beta  (u_t,u) - \alpha \beta ||u||^2 + \alpha \beta (v,u) + \beta ||u_t ||^2 \\
&& -\beta c^2 ||\nabla v ||^2 - \varepsilon_v \beta  (u_t,u) -
\alpha \beta ||v||^2 + \alpha \beta (u,v) + \beta ||v_t ||^2.
\end{array}
\end{equation}
\item $L$ is equivalent to $E$: There exists two positive real numbers  $K_m$ and $K_M$ such that $K_m E(t) \le L(t) \le K_M E(t)$.\\
Indeed, by using the Cauchy-Schwarz, Poincarr\'e (with constant
$S_p$)  and Young inequalities, we get
\[
\begin{array}{rcl}
\medskip
|L(t)- NE(t)| &=& |f(t)|\\ \medskip
&\le&  \displaystyle \frac{\beta}{2} ||u_t ||^2 + \frac{\beta S_p^2}{2} ||\nabla u ||^2  + \displaystyle \frac{\beta}{2} ||v_t ||^2 + \displaystyle \frac{\beta S_p^2}{2} ||\nabla v ||^2\\
&\le& \beta \min(1, c^2 S_p^2 ) E(t).
\end{array}
\]
For $N$ sufficiently large , we easily deduce the existence of
positive constants $K_m$ and  $K_M$ such that
\begin{equation}\label{equiv}
K_m E(t) \le L(t) \le K_M E(t).
\end{equation}
\item $L'(t) \le - \hat{K} L(t)$ for a given positive constant $\hat{K}$:\\
We consider Relation \eqref{Lprim}, \eqref{Epr} and \eqref{frp}.
We get by using the Cauchy-Schwarz inequality, the Young  relation
$ ab \le \displaystyle \frac{1}{2\varepsilon_i} a^2 +
\frac{\varepsilon_i}{2} b^2$ for the terms $(u_t,u)$ and
$(v_t,v)$, the relation $ -(a-b)^2=-a^2 -b^2 +2ab$, and the
Poincar\'e inequality, the following relation:
\[
\begin{array}{rcl}
\medskip
L'(t) &\le&  \displaystyle (\beta - N \varepsilon_u +
\frac{\varepsilon_u \beta}{2\varepsilon_1}) ||u_t ||^2 + (-\beta
c^2 + \frac{\varepsilon_u \beta S_p^2 \varepsilon_1}{2} ) ||\nabla
u ||^2  )\\ \medskip && \displaystyle (\beta - N \varepsilon_v +
\frac{\varepsilon_v \beta}{2\varepsilon_2}) ||v_t ||^2 + (-\beta
c^2 + \frac{\varepsilon_v \beta S_p^2 \varepsilon_2}{2} ) ||\nabla
u ||^2  ) -\alpha \beta ||u-v||^2.
\end{array}
\]
We choose $\varepsilon_1 < \displaystyle \frac{2 c^2}{
\varepsilon_u S_p^2} $, $\varepsilon_2 < \displaystyle \frac{2
c^2}{ \varepsilon_v S_p^2} $ and $N > \displaystyle  \beta
\max(\frac{1}{\varepsilon_u} + \frac{1}{2\varepsilon_1},
\frac{1}{\varepsilon_v} +
 \frac{1}{2\varepsilon_2})$, and we deduce that there exist a positive constant $\bar{K}$ such that
\[
L'(t) \le - \bar{K} \big(  ||u_t ||^2 + ||v_t ||^2 + c^2 ||\nabla
u ||^2 + c^2 ||\nabla v ||^2 + \alpha ||u - v ||^2 \big).
\]
Hence, we get the relation
\begin{equation}\label{relaL}
L'(t) \le - \hat{K} L(t)
\end{equation}
with $\hat{K}= \displaystyle \frac{2 \bar{K}}{K_m}$.
\item Exponential decay of the energy:\\
Relation \eqref{relaL} deduces the following:
$$ L(t) \le e^{-\hat{K} t} L(0).$$
The relation \eqref{equiv} allows us to deduce the following
exponential decay of the energy
$$E(t) \le \displaystyle \frac{K_M}{K_m} e^{-\hat{K} t} E(0), $$
and then the desired result.
\end{enumerate}
\end{proof}
\section{Finite element method and discrete scheme}\label{chpred}
In order to approximate the solution's of System \eqref{PDE}, we use the finite element method for the space discretizations and the finite differences method for the time discretizations. From now on, we assume that $\Omega$ is a polygon when $d=2$ or polyhedron when $d=3$, so it can be completely meshed.\\

\noindent For the space discretizations,  we consider a regular
(see Ciarlet~\cite{PGC}) family of triangulations $( \mathcal{T}_h
)_h$ of $\Omega$ which is a set of closed, non-degenerate
triangles for $d=2$ or tetrahedra for $d=3$, called elements,
satisfying the following properties:
\begin{itemize}
\item for each $h$, $\bar{\Omega}$ is the union of all elements in
 $\mathcal{T}_h$;
\item the intersection of two distinct elements of
$\mathcal{T}_h$ is either empty, a common vertex, or an entire
common edge (or face when $d=3$);
\item the ratio of the diameter {  $h_{\kappa}$  of an element
$\kappa \in \mathcal{T}_h$  to the diameter  $\rho_\kappa$}  of
its inscribed circle when $d=2$ or ball when $d=3$ is bounded by a
constant independent of $h$. That is, there exists a strictly
positive constant $\sigma$, independent of $h$, such that,
\begin{equation}
\label{eq:reg} { \displaystyle \max_{\kappa \in \mathcal{T}_{h}}
\frac{h_{ \kappa}}{\rho_{ \kappa}} \le \sigma.}
\end{equation}
\end{itemize}
As usual, $h$ denotes the maximal diameter of all elements of $\mathcal{T}_{h}$. \\
%
\noindent To define the finite element functions, let $r$ be a non-negative integer. For each {  $\kappa$ in $\mathcal{T}_{h}$, we denote by  $\mathbb{P}_r(\kappa)$  the space of restrictions to $\kappa$ of polynomials in $d$ variables and total degree at most  $r$.\\
\noindent We approximate the unkows $u$ and $v$ of Problem
\eqref{PDE}  in the spaces $X_{0h}$ defined as
\[
X_{0h}=\left\{v_h\in  {\cal C}^0(\overline{\Omega}),~\text{for
all}~\kappa\in {\cal T}_h,~~v_h|_\kappa\in
\mathbb{P}_1(\kappa)\right\} \bigcap H^1_0(\Omega) .
\]
%
\noindent For time discretization, we devide the interval $[0,T]$ into $M$ segments and we denote by $k$ the time step and by $t_k=nk, k=0, \dots , M.$\\

\noindent The problem \eqref{PDE} can be approximated by the following iterative scheme: \\
Let $u^0_h$ and $v_h^0$ be the approximations of $u_0(x)$ and $v_0(x)$ in $X_{0h}$. We also need to approximate $u_h^1$ and $v_h^1$. Thus, we consider $\bar{u}_h^1$ and $\bar{v}_h^1$ approximations of $u_1(x)$ and $v_1(x)$ in $X_{0h}$ and take $u_h^1=u_h^0 + k \bar{u}_h^1$ and $v_h^1=v_h^0 + k \bar{v}_h^1$. \\
\noindent For each iteration $n$ and having $u_h^{n-1},u_h^n,
v_h^{n-1}$ and $v_h^n$ in $X_{0h}$, compute $(u_h^n,v_h^n)\in
X_{0h}^2$ such that for all $(\phi_h,\psi_h)\in X_{0h}^2$,
\begin{equation}\label{PDEd}
\left\{
\begin{array}{ll}
\medskip
\displaystyle (\frac{u_h^{n+1}- 2 u_h^n + u_h^{n-1}}{k^2}, \phi_h)
+ c^2 (\nabla u_h^{n+1}, \nabla \phi_h) \\ \medskip
\hspace{4cm}+ \displaystyle \varepsilon_u (\frac{u_h^{n+1}-u_h^n}{k}, \phi_h) +\alpha (u_h^{n+1}, \phi_h) - \alpha (v_h^{n+1}, \phi_h)=0, \\
\medskip
\displaystyle (\frac{v_h^{n+1}- 2 v_h^n + v_h^{n-1}}{k^2}, \psi_h) + c^2 (\nabla v_h^{n+1}, \nabla \psi_h)\\
\hspace{4cm}+ \displaystyle \varepsilon_v
(\frac{v_h^{n+1}-v_h^n}{k}, \psi_h) +\alpha (v_h^{n+1}, \psi_h) -
\alpha (u_h^{n+1}, \psi_h)=0.
\end{array}
\right.
\end{equation}
\begin{thm}
For each iteration $n$ and for each $(u_h^{n-1}, u_h^n,
v_h^{n-1},v_h^n) \in X_{0h}^4$, problem \eqref{PDE} admits a
unique solution $(u_h^n,v_h^n) \in X_{0h}^2$.
\end{thm}
\begin{proof}  Since the problem \eqref{PDEd} is given in a finite-dimensional space $X_{0h}$, it helps to prove the uniqueness of the solution which also leads to the existence.\\
\noindent Let $u_h^{n-1}, u_h^n, v_h^{n-1}$ and $v_h^n$ be given
elements in $X_{0h}$. At each iteration $n$, let $u_{h1}^{n+1}$
and $v^n_{h1}$ be two solutions of Problem \eqref{PDEd} and denote
by $u^{n+1}_h= u^{n+1}_{h1} - u^{n+1}_{h2}$ and $v^{n+1}_h=
v^{n+1}_{h1} - v^{n+1}_{h2}$. Thus, $u_h^{n+1}$ and $v_h^{n+1}$
satisfy the following equations:
\begin{equation}\label{PDEdu}
\left\{
\begin{array}{ll}
\medskip
\displaystyle (\frac{u_h^{n+1}}{k^2}, \phi_h) + c^2 (\nabla u_h^{n+1}, \nabla \phi_h) + \displaystyle \varepsilon_u (\frac{u_h^{n+1}}{k}, \phi_h) +\alpha (u_h^{n+1}, \phi_h) - \alpha (v_h^{n+1}, \phi_h)=0, \\
\medskip
\displaystyle (\frac{v_h^{n+1}}{k^2}, \psi_h) + c^2 (\nabla
v_h^{n+1}, \nabla \psi_h)+ \displaystyle \varepsilon_v
(\frac{v_h^{n+1}}{k}, \psi_h) +\alpha (v_h^{n+1}, \psi_h) - \alpha
(u_h^{n+1}, \psi_h)=0.
\end{array}
\right.
\end{equation}
Taking $\phi_h= u_h^{n+1}$ and $\psi_h=v_h^{n+1}$, and summing the
above two equations get the following:
\[
\begin{array}{ll}
\medskip
\displaystyle \frac{1}{k^2} ||u_h^{n+1}||^2 + c^2 ||\nabla u_h^{n+1}||^2 + \displaystyle \frac{\varepsilon_u}{k} ||u_h^{n+1}||^2 +\frac{1}{k^2} ||v_h^{n+1}||^2 + c^2 ||\nabla v_h^{n+1}||^2 + \displaystyle \frac{\varepsilon_v}{k} ||v_h^{n+1}||^2\\
\hspace{10cm}+\alpha ||u_h^{n+1} - v_h^{n+1}||^2 =0.
\end{array}
\]
he positivity of all the terms on the left-hand side of the last
equation implies that  $u_h^{n+1}=v_h^{n+1}=0$, and consequently,
$(u^{n+1}_{h1}, v^{n+1}_{h1})=(u^{n+1}_{h2}, v^{n+1}_{h2})$. Thus,
we obtain the uniqueness and existence of the solution
$(u^{n+1}_{h}, v^{n+1}_{h})$ of the problem \eqref{PDEd}.
\end{proof}
\section{A priori error analysis}

The objective of this section is to establish an a priori error
estimate. There exists an approximation operator (when $d=2$, see
Bernardi and Girault~\cite{s5} or Cl\'ement~\cite{s3}; when $d =2$
or $d =3$, see Scott and Zhang~\cite{ZC}), $R_h$ in ${\mathcal
L}(W^{1,p}(\Omega)\cap H^1_0(\Omega); X_{0h})$ such that for all
$m=0,1$, $l=0,1$, and all $p\ge 2$,
\begin{equation}\label{proj}
 \forall\, S \in
W^{l+1,p}(\Omega),\,\,|S-R_h(S)|_{W^{m,p}(\Omega)}\leq C(p,m,l)
\,h^{l+1-m}|S|_{W^{l+1,p}(\Omega)},
\end{equation}
where $C$ is a positive constant independent of $h$ and $S$.
\begin{thm}\label{thmapr}
If $u,v$ are the solution of the problem \eqref{PDE} and
$u_h^{n+1}, v_h^{n+1}, 1\le n\le  M-1,$ are the solutions of the
iterative scheme \eqref{PDEd} , then  for all
$(\overline{u}_h^{n+1},\overline{v}_h^{n+1})\in S_h\times S_h$,
the following error estimate holds
\begin{equation}\label{error-1}
\begin{array}{ll}
\medskip
\displaystyle
\sup_{n<M}\left\{\|\hat{u}^{n+1}-\hat{u}_h^{n+1}\|^2+\|\hat{v}^{n+1}-\hat{v}_h^{n+1}\|^2+\|\nabla(u^{n+1}-u_h^{n+1})\|^2
\right. \\ \medskip
\displaystyle \hspace{4cm} \left.
+\|\nabla(v^{n+1}-v_h^{n+1})\|^2+\|(u^{n+1}-v^{n+1})-(u_h^{n+1}-v_h^{n+1})\|^2\right\}
\\\medskip \displaystyle
\leq Ck\sum_{j=0}^{N-1}\Bigg(\|\hat{u}_t^{j+1}-\delta \hat{u}^{j+1}\|^2+\|\hat{u}^{j+1}-\overline{u}_h^{j+1}\|^2+\|\hat{v}_t^{j+1}-\delta \hat{v}^{j+1}\|^2+\|\hat{v}^{j+1}-\overline{v}_h^{j+1}\|^2\\
\hspace{.5cm} \medskip  +\|\nabla(\hat{u}^{j+1}-\overline{u}_h^{j+1})\|^2+\|\nabla(\hat{v}^{j+1}-\overline{v}_h^{j+1})\|^2+\|\nabla(\hat{u}^{j+1}-\delta u^{j+1})\|^2+\|\nabla(\hat{v}^{j+1}-\delta v^{j+1})\|^2\Bigg)\\
\hspace{.5cm} \medskip \displaystyle +\frac{C}{k}\sum_{j=0}^{N-2}\Bigg(\|(\hat{u}^{j+1}-\overline{u}_h^{j+1})-(\hat{u}^{j+2}-\overline{u}_h^{j+2})\|^2+\|(\hat{v}^{j+1}-\overline{v}_h^{j+1})-(\hat{v}^{j+2}-\overline{v}_h^{j+2})\|^2\Bigg)\\
\hspace{.5cm} \medskip \displaystyle +C\sup_{n}(\|\hat{u}^{n+1}-\overline{u}_h^{n+1}\|^2+\|\hat{u}_h^0-\hat{u}^0\|^2+\|\hat{u}^1-\overline{u}_h^1\|^2)\\
\hspace{.5cm} \displaystyle
+C\sup_{n}(\|\hat{v}^{n+1}-\overline{v}_h^{n+1}\|^2+\|\hat{v}_h^0-\hat{v}^0\|^2+\|\hat{v}^1-\overline{v}_h^1\|^2),
\end{array}
\end{equation}
where $C$ is positive constant independent oh $h$, $k$ and
$(u,v)$.
\end{thm}
\begin{proof}
The proof consists of several steps. We start by rewriting the
scheme \eqref{PDEd} as follows:
\begin{equation}\label{PDEdd}
\left\{
\begin{array}{ll}
\medskip
\displaystyle (\frac{\hat{u}_h^{n+1}- \hat{u}_h^{n}}{k},
\overline{u}) + c^2 (\nabla u_h^{n+1}, \nabla \overline{u}) \\
\medskip
\hspace{4cm}+ \displaystyle \varepsilon_u (\hat{u}_h^{n+1}, \overline{u}) +\alpha (u_h^{n+1}, \overline{u}) - \alpha (v_h^{n+1}, \overline{u})=0, \\
\medskip
\displaystyle (\frac{\hat{v}_h^{n+1}- \hat{v}_h^{n}}{k}, \overline{v}) + c^2 (\nabla v_h^{n+1}, \nabla \overline{v})\\
\hspace{4cm}+ \displaystyle \varepsilon_v (\hat{v}_h^{n+1},
\overline{v}) +\alpha (v_h^{n+1}, \overline{v}) - \alpha
(u_h^{n+1}, \overline{v})=0,
\end{array}
\right.
\end{equation}
where $u_h^{n+1}=u_h^n+k \hat{u}_h^{n+1}$ and $v_h^{n+1}=v_h^n+k  \hat{v}_h^{n+1}$.\\
\noindent
In addition, we introduce the following weak form of \eqref{PDE}:\\
\begin{equation}\label{PDEW}
\left\{
\begin{array}{ll}
\medskip
\displaystyle (u_{tt}, \overline{u}) + c^2 (\nabla u, \nabla \overline{u}) +\varepsilon_u (u_t, \overline{u}) +\alpha (u-v, \overline{u})=0, \\
\medskip
\displaystyle (v_{tt}, \overline{v}) + c^2 (\nabla v, \nabla
\overline{v})  + \varepsilon_v (v_t, \overline{v}) +\alpha (v-u,
\overline{v})=0.
\end{array}
\right.
\end{equation}
Later, we denote by $\hat{u}=u_t$ and $\hat{u}^{n+1}=u_t(t_{n+1})$.\\

\noindent \textit{\underline{Step 1}}: \\
By substructing \eqref{PDEdd}$_1$ and \eqref{PDEW}$_1$, we get for
all $\overline{u}_h$ and $\overline{v}_h$ in $X_h$,
\begin{equation}\label{first}
\begin{array}{ll}
\medskip
\displaystyle
\left(\hat{u}_t^{n+1}-\frac{1}{k}\left[\hat{u}_h^{n+1}-\hat{u}_h^n\right],\overline{u}_h\right)+c^2\left(\nabla
u^{n+1}-\nabla u_h^{n+1},\nabla \overline{u}_h\right) \\ \medskip
\hspace{1cm}+ \displaystyle
\varepsilon_u\left(\hat{u}^{n+1}-\hat{u}_h^{n+1},\overline{u}_h\right)
+ \alpha (u^{n+1}-v^{n+1}-\left[u_h^{n+1}-v_h^{n+1}\right],
\overline{u}_h)=0.
\end{array}
 \end{equation}
By adding and substructing $\hat{u}^{n+1}$, taking into
consideration the above identity, taking
$\overline{u}_h=\hat{u}_h^{n+1}$ in \eqref{PDEdd}, we get :
\begin{equation}\label{first-1-u}
\begin{array}{ll}
\medskip
\displaystyle
\left(\hat{u}_t^{n+1}-\frac{1}{k}\left[\hat{u}_h^{n+1}-\hat{u}_h^n\right],\hat{u}^{n+1}-\hat{u}_h^{n+1}\right)+c^2\left(\nabla(u^{n+1}-u_{h}^{n+1}),\nabla(\hat{u}^{n+1}-\hat{u}_h^{n+1})\right)

\\ \medskip
\hspace{0cm} \displaystyle
+\varepsilon_u\left(\hat{u}^{n+1}-\hat{u}_h^{n+1},\hat{u}^{n+1}-\hat{u}_h^{n+1}\right)+\alpha\left((u^{n+1}-u_h^{n+1})-(v^{n+1}-v_h^{n+1}),
\hat{u}^{n+1}-\hat{u}_h^{n+1}\right) \\ \medskip
\hspace{0cm} \displaystyle =
\left(\hat{u}_t^{n+1}-\frac{1}{k}\left[\hat{u}_h^{n+1}-\hat{u}_h^n\right],\hat{u}^{n+1}-\overline{u}_h\right)+c^2\left(\nabla(u^{n+1}-u_{h}^{n+1}),\nabla(\hat{u}^{n+1}-\overline{u}_h)\right)\\
\medskip
\hspace{0cm} \displaystyle
+\varepsilon_u\left(\hat{u}^{n+1}-\hat{u}_h^{n+1},\hat{u}^{n+1}-\overline{u}_h\right)+\alpha\left((u^{n+1}-u_h^{n+1})-(v^{n+1}-v_h^{n+1}),
\hat{u}^{n+1}-\overline{u}_h\right).
\end{array}
\end{equation}
In a similar we get
\begin{equation}\label{first-1-v}
\begin{array}{ll}
\medskip
\displaystyle
\left(\hat{v}_t^{n+1}-\frac{1}{k}\left[\hat{v}_h^{n+1}-\hat{v}_h^n\right],\hat{v}^{n+1}-\hat{v}_h^{n+1}\right)+c^2\left(\nabla(v^{n+1}-v_{h}^{n+1}),\nabla(\hat{v}^{n+1}-\hat{v}_h^{n+1})\right)

\\ \medskip
\hspace{0cm} \displaystyle
+\varepsilon_v\left(\hat{v}^{n+1}-\hat{v}_h^{n+1},\hat{v}^{n+1}-\hat{v}_h^{n+1}\right)+\alpha\left((v^{n+1}-v_h^{n+1})-(u^{n+1}-u_h^{n+1}),
\hat{v}^{n+1}-\hat{v}_h^{n+1}\right) \\ \medskip
\hspace{0cm} \displaystyle =
\left(\hat{v}_t^{n+1}-\frac{1}{k}\left[\hat{v}_h^{n+1}-\hat{v}_h^n\right],\hat{v}^{n+1}-\overline{v}_h\right)+c^2\left(\nabla(v^{n+1}-v_{h}^{n+1}),\nabla(\hat{v}^{n+1}-\overline{v}_h)\right)\\
\medskip
\hspace{0cm} \displaystyle
+\varepsilon_v\left(\hat{v}^{n+1}-\hat{v}_h^{n+1},\hat{v}^{n+1}-\overline{v}_h\right)+\alpha\left((v^{n+1}-v_h^{n+1})-(u^{n+1}-u_h^{n+1}),
\hat{v}^{n+1}-\overline{v}_h\right).
\end{array}
\end{equation}
%

\noindent \textit{\underline{Step 2}}: \\
By noting that
\begin{equation}\label{chap-1}
\begin{array}{ll}
\medskip
\displaystyle
\left(\hat{u}_t^{n+1}-\frac{1}{k}\left[\hat{u}_h^{n+1}-\hat{u}_h^n\right],\hat{u}^{n+1}-\hat{u}_h^{n+1}\right)
%
\\ \medskip
=
\displaystyle\left(\hat{u}_t^{n+1}-\frac{1}{k}\left[\hat{u}^{n+1}-\hat{u}^n\right],\hat{u}^{n+1}-\hat{u}_h^{n+1}\right)
+\frac{1}{2k}\|\hat{u}^{n+1}-\hat{u}_h^{n+1}-(\hat{u}^n-\hat{u}_h^n)\|^2
\\ \medskip
\hspace{2cm}\displaystyle
+\frac{1}{2k}\left[\|\hat{u}^{n+1}-\hat{u}_h^{n+1}\|^2-\|\hat{u}^n-\hat{u}_h^n\|^2\right],
\end{array}
\end{equation}
 deriving similar identity for $v$ and adding \eqref{first-1-u}  and \eqref{first-1-v} , neglecting some positive terms leads to

\begin{equation}\label{chap-2}
\begin{array}{ll}
\medskip
\displaystyle
\frac{1}{2k}\|\hat{u}^{n+1}-\hat{u}_h^{n+1}-(\hat{u}^n-\hat{u}_h^n)\|^2+\frac{1}{2k}\left[\|\hat{u}^{n+1}-\hat{u}_h^{n+1}\|^2
-\|\hat{u}^n-\hat{u}_h^n\|^2\right]
\\ \medskip
\hspace{.5cm}\displaystyle
+\left(\hat{u}_t^{n+1}-\frac{1}{k}[\hat{u}^{n+1}-\hat{u}^n],
\hat{u}^{n+1}-\hat{u}_h^{n+1}\right)+c^2\left(\nabla(u^{n+1}-u_h^{n+1}),\nabla(\hat{u}^{n+1}-\hat{u}_h^{n+1})\right)\\
\medskip
\hspace{.5cm}\displaystyle
+\varepsilon_u\|\hat{u}^{n+1}-\hat{u}_h^{n+1}\|^2+\alpha\left((u^{n+1}-u)-(v^{n+1}-v_h^{n+1}),(\hat{u}^{n+1}-\hat{u}_h^{n+1})-(\hat{v}^{n+1}-\hat{v}_h^{n+1})\right)\\
\medskip
\hspace{.5cm}\displaystyle
+\frac{1}{2k}\|\hat{v}^{n+1}-\hat{v}_h^{n+1}-(\hat{v}^n-\hat{v}_h^n)\|^2+\frac{1}{2k}\left[\|\hat{v}^{n+1}-\hat{v}_h^{n+1}\|^2-\|\hat{v}^{n}-\hat{v}_h^{n}\|^2\right]
\\ \medskip
\hspace{.5cm}\displaystyle
+\left(\hat{v}_t^{n+1}-\frac{1}{k}[\hat{v}^{n+1}-\hat{v}^n],
\hat{v}^{n+1}-\hat{v}_h^{n+1}\right)+c^2\left(\nabla(v^{n+1}-v_h^{n+1}),\nabla(\hat{v}^{n+1}-\hat{v}_h^{n+1})\right)
\\ \medskip
\displaystyle \le
\left(\hat{u}_t^{n+1}-\frac{1}{k}[\hat{u}_h^{n+1}-\hat{u}_h^n],
\hat{u}^{n+1}-\overline{u}_h\right)+c^2\left(\nabla(u^{n+1}-u_h^{n+1}),\nabla(\hat{u}^{n+1}-\overline{u}_h)\right)\\
\medskip \hspace{.5cm}\displaystyle
+\varepsilon_u(\hat{u}^{n+1}-\hat{u}_h^{n+1},\hat{u}^{n+1}-\overline{u}_h)
+\varepsilon_v(\hat{v}^{n+1}-\hat{v}_h^{n+1},\hat{v}^{n+1}-\overline{v}_h)\\
\medskip \hspace{1cm}\displaystyle +\alpha
\left((u^{n+1}-u_h^{n+1})-(v^{n+1}-v_h^{n+1}),(\hat{u}^{n+1}-\overline{u}_h)-(\hat{v}^{n+1}-\overline{v}_h)\right)\\
\medskip \hspace{.5cm}\displaystyle
+\left(\hat{v}_t^{n+1}-\frac{1}{k}\left[\hat{v}_h^{n+1}-\hat{v}_h^n\right],
\hat{v}^{n+1}-\overline{v}_h\right)+c^2\left(\nabla(v^{n+1}-v_h^{n+1}),\nabla(\hat{v}^{n+1}-\overline{v}_h)\right).
\end{array}
\end{equation}

\noindent \underline{\textit{Step 3}}\\
Using the definitions of $\hat{u}_h^{n+1}$ and $\hat{v}_h^{n+1}$,
the following identities are valid,
\begin{equation}\label{u-chap}
\begin{array}{ll}
    \medskip
\displaystyle
\left(\nabla(u^{n+1}-u_h^{n+1}),\nabla(\hat{u}^{n+1}-\hat{u}_h^{n+1})\right)=\left(\nabla(u^{n+1}-u_h^{n+1}),\nabla(\hat{u}^{n+1}-\frac{u^{n+1}-u^n}{k})\right)\\
\medskip \hspace{1cm}\displaystyle
+\frac{1}{2k}\left[\|\nabla(u^{n+1}-u_h^{n+1})-\nabla(u^n-u_h^n)\|^2\right]
+\frac{1}{2k}\left[\|\nabla(u^{n+1}-u_h^{n+1})\|^2-\|\nabla(u^n-u_h^n)\|^2\right],
\end{array}
\end{equation}
\begin{equation}\label{v-chap}
\begin{array}{ll}
    \medskip
\displaystyle
\left(\nabla(v^{n+1}-v_h^{n+1}),\nabla(\hat{v}^{n+1}-\hat{v}_h^{n+1})\right)=\left(\nabla(v^{n+1}-v_h^{n+1}),\nabla(\hat{v}^{n+1}-\frac{v^{n+1}-v^n}{k})\right)\\
\medskip \hspace{1cm}\displaystyle
+\frac{1}{2k}\left[\|\nabla(v^{n+1}-v_h^{n+1})-\nabla(v^n-v_h^n)\|^2\right]
+\frac{1}{2k}\left[\|\nabla(u^{n+1}-u_h^{n+1})\|^2-\|\nabla(u^n-u_h^n)\|^2\right],
\end{array}
\end{equation}
and
\begin{equation}\label{u-v-chap}
\begin{array}{ll}
    \medskip
\displaystyle
\left((u^{n+1}-v^{n+1})-(u_h^{n+1}-v_h^{n+1}), (\hat{u}^{n+1}-\hat{v}^{n+1})-(\hat{u}_h^{n+1}-\hat{v}_h^{n+1})\right)=\\
\medskip \displaystyle \hspace{3cm} \left((u^{n+1}-v^{n+1})-(u_h^{n+1}-v_h^{n+1}), (\hat{u}^{n+1}-\hat{v}^{n+1})-\frac{(u^{n+1}-v^{n+1})-(u^n-v^n)}{k}\right)\\
\medskip\displaystyle \hspace{3cm}+\frac{1}{2k}\|[(u^{n+1}-u_h^{n+1})-(v^{n+1}-v_h^{n+1})]-[(u^n-u_h^n)-(v^n-v_h^n)]\|^2\\
\medskip\displaystyle\hspace{3cm}
+\frac{1}{2k}\left[\|(u^{n+1}-u_h^{n+1})-(v^{n+1}-v_h^{n+1})\|^2-\|(u^n-u_h^n)-(v^n-v_h^n)\|^2\right].
\end{array}
\end{equation}
Combining  with \eqref{u-chap}, \eqref{v-chap} and
\eqref{u-v-chap} with \eqref{chap-2} , and applying Cauchy Shwartz
and Young inequality,  give
\begin{equation}\label{chap-4}
\begin{array}{ll}
\medskip
\displaystyle
\frac{1}{2k}\left[\|\hat{u}^{n+1}-\hat{u}_h^{n+1}\|^2
-\|\hat{u}^n-\hat{u}_h^n\|^2\right]
+\frac{1}{2k}\left[\|\hat{v}^{n+1}-\hat{v}_h^{n+1}\|^2
-\|\hat{v}^n-\hat{v}_h^n\|^2\right]
\\ \medskip
\hspace{0.5cm}\displaystyle
+\frac{c^2}{2k}\left[\|\nabla(u^{n+1}-u_h^{n+1})\|^2-\|\nabla(u^n-u_h^n)\|^2\right]
+\frac{c^2}{2k}\left[\|\nabla(v^{n+1}-v_h^{n+1})\|^2-\|\nabla(v^n-v_h^n)\|^2\right]
\\ \medskip
\hspace{0.5cm}\displaystyle
+\frac{\alpha}{2k}\left[\|(u^{n+1}-v^{n+1})-(u_h^{n+1}-v_h^{n+1})\|^2
-\|(u^{n}-v^{n})-(u_h^{n}-v_h^{n})\|^2\right]
\\ \medskip
\hspace{0.cm}\displaystyle \leq  C
\left(\hat{u}_t^{n+1}-\frac{1}{k}[\hat{u}_h^{n+1}-\hat{u}_h^n], \hat{u}^{n+1}-\overline{u}_h\right)+\left(\hat{v}_t^{n+1}-\frac{1}{k}[\hat{v}_h^{n+1}-\hat{v}_h^n], \hat{v}^{n+1}-\overline{v}_h\right)\\
\medskip
\displaystyle\hspace{1cm} \frac{1}{k}\left([\hat{u}^{n+1}-\hat{u}^n]-[\hat{u}_h^{n+1}-\hat{u}_h^n],\hat{u}^{n+1}-\overline{u}_h\right)+\frac{1}{k}\left([\hat{v}^{n+1}-\hat{v}^n]-[\hat{v}_h^{n+1}-\hat{v}_h^n],\hat{v}^{n+1}-\overline{v}_h\right)\\
\medskip\displaystyle\hspace{1cm} +\|\hat{u}^{n+1}-\hat{u}_h^{n+1}\|^2+\|\hat{u}^{n+1}-\overline{u}_h\|^2+\|\hat{v}^{n+1}-\hat{v}_h^{n+1}\|^2+\|\hat{v}^{n+1}-\overline{v}_h\|^2\\
\medskip\hspace{1cm} +\|\nabla(u^{n+1}-u_h^{n+1})\|^2+\|\nabla(\hat{u}^{n+1}-\overline{u}_h)\|^2+\|\nabla(v^{n+1}-v_h^{n+1})\|^2+\| \nabla(\hat{v}^{n+1}-\overline{v}_h)\|^2\\
\medskip\hspace{0.5cm} \|(u^{n+1}-u_h^{n+1})-(v^{n+1}-v_h^{n+1}))\|^2+\|(\hat{u}^{n+1}-\overline{u}_h)-(v^{n+1}-\overline{v}_h)\|^2)
\\ \medskip
\hspace{1cm}\displaystyle
+\left(\hat{u}_t^{n+1}-\frac{1}{k}[\hat{u}^{n+1}-\hat{u}^n],
\hat{u}^{n+1}-\hat{u}_h^{n+1}\right)+\left(\hat{v}_t^{n+1}-\frac{1}{k}[\hat{v}^{n+1}-\hat{v}^n],
\hat{v}^{n+1}-\hat{v}_h^{n+1}\right).

\end{array}
\end{equation}
%
\noindent \underline{\textit{Step 4}}\\
We define the operator $\delta v = \displaystyle \frac{1}{k}(v^{n+1}-v^n)$. \\
\noindent By noting that $$ \left(\hat{u}_t^{n+1}-\delta
\hat{u}_h^{n+1},\hat{u}^{n+1}-\overline{u}_h\right)\leq
C\|\hat{u}^{n+1}-\delta\hat{u}^{n+1}\|^2+C\|\hat{u}^{n+1}-\overline{u}_h\|^2+\left(\delta
\hat{u}^{n+1}-\delta
\hat{u}_h^{n+1},\hat{u}^{n+1}-\overline{u}_h\right),
$$
and using Young inequality with $\varepsilon$ small enough, we get

\begin{equation}\label{chap-5}
\begin{array}{ll}
\medskip
\displaystyle
\frac{1}{2k}\left[\|\hat{u}^{n+1}-\hat{u}_h^{n+1}\|^2
-\|\hat{u}^n-\hat{u}_h^n\|^2\right]
+\frac{1}{2k}\left[\|\hat{v}^{n+1}-\hat{v}_h^{n+1}\|^2
-\|\hat{v}^n-\hat{v}_h^n\|^2\right]
\\ \medskip
\hspace{1cm}\displaystyle
+\frac{c^2}{2k}\left[\|\nabla(u^{n+1}-u_h^{n+1})\|^2-\|\nabla(u^n-u_h^n)\|^2\right]
+\frac{c^2}{2k}\left[\|\nabla(v^{n+1}-v_h^{n+1})\|^2-\|\nabla(v^n-v_h^n)\|^2\right]

\\ \medskip
\hspace{1cm}\displaystyle
+\frac{\alpha}{2k}\left[\|(u^{n+1}-v^{n+1})-(u_h^{n+1}-v_h^{n+1})\|^2
-\|(u^{n}-v^{n})-(u_h^{n}-v_h^{n})\|^2\right]
\\ \medskip
\hspace{0cm}\displaystyle \leq C\|\hat{u}_t^{n+1}-\delta
\hat{u}^{n+1}\|^2+C\|\hat{u}^{n+1}-\overline{u}_h\|^2+C\|\hat{v}_t^{n+1}-\delta
\hat{v}^{n+1}\|^2+C\|\hat{v}^{n+1}-\overline{v}_h\|^2\\ \medskip
\hspace{1cm}\displaystyle +(\delta \hat{u}^{n+1}-\delta
\hat{u}_h^{n+1}, \hat{u}^{n+1}-\overline{u}_h)+(\delta
\hat{v}^{n+1}-\delta \hat{v}_h^{n+1},
\hat{v}^{n+1}-\overline{v}_h)\\ \medskip \hspace{1cm}\displaystyle
+C\varepsilon
\|\hat{u}^{n+1}-\hat{u}_h^{n+1}\|^2+C\|\hat{u}^{n+1}-\overline{u}_h\|^2+C\varepsilon\|\hat{v}^{n+1}-\hat{v}_h^{n+1}\|^2+C\|\hat{v}^{n+1}-\overline{v}_h\|^2\\
\medskip \hspace{1cm}\displaystyle +(\delta \hat{u}^{n+1}-\delta
\hat{u}_h^{n+1},\hat{u}^{n+1}-\overline{u}_h)+(\delta
\hat{v}^{n+1}-\delta
\hat{v}_h^{n+1},\hat{v}^{n+1}-\overline{v}_h)\\ \medskip
\hspace{1cm}\displaystyle
C\|\nabla(u^{n+1}-u_h^{n+1})\|^2+\|\nabla(\hat{u}^{n+1}-\overline{u})\|^2+C\|\nabla(v^{n+1}-v_h^{n+1})\|^2+\|\nabla(\hat{v}^{n+1}-\overline{v})\|^2\\
\medskip \hspace{1cm}\displaystyle
+C\|(u^{n+1}-v^{n+1})-(u_{h}^{n+1}-v_h^{n+1})\|^2+C\|\nabla(\hat{u}^{n+1}-\delta
u^{n+1})\|^2+C\|\nabla(\hat{v}^{n+1}-\delta v^{n+1})\|^2.
\end{array}
\end{equation}

\noindent \textit{\underline{Step 5}}: \\
As in \cite{Copetti}, the following inequality holds for all
$\overline{u}_h^j\in X_h$,
\begin{equation}
\begin{array}{ll}
\medskip
\displaystyle k\sum_{j=0}^n (\delta \hat{u}^{j+1}-\delta
\hat{u}_h^{j+1}, \hat{u}^{j+1}-\overline{u}_h^{j+1})\\ \medskip
= \displaystyle (\hat{u}^{n+1}-\hat{u}_h^{n+1},\hat{u}^{n+1}-\overline{u}_h^{n+1})+(\hat{u}_h^0-\hat{u}^0,\hat{u}^1-\overline{u}_h^1) \\
\medskip
\hspace{1cm} + \displaystyle k\sum_{j=0}^{n-1}(\hat{u}^{j+1}-\hat{u}_h^{j+1},\hat{u}^{j+1}-\overline{u}_h^{j+1}-(\hat{u}^{j+2}-\overline{u}_h^{j+2}))\\
\leq \displaystyle C\|\hat{u}^{n+1}-\hat{u}_h^{n+1}\|^2+C\|\hat{u}^{n+1}-\overline{u}_h^n\|^2+\|\hat{u}_h^0-\hat{u}^0\|^2  +\|\hat{u}^1-\overline{u}_h^1\|^2+Ck\sum_{j=1}^{n-2}\|\hat{u}^{j+1}-\hat{u}_h^{j+1}\|^2\\
\hspace{1cm} \displaystyle
+\frac{C}{k}\sum_{j=1}^{n-2}\|\hat{u}^{j+1}-\overline{u}_h^{j+1}-(\hat{u}^{j+2}-\overline{u}_h^{j+2})\|^2.
\end{array}
\end{equation}
Now, letting
$e_n=\|\hat{u}^n-\hat{u}_h^n\|^2+\|\hat{v}^n-\hat{v}_h^n\|^2+\|\nabla(u^n-u_h^n)\|^2+\|\nabla(v^n-v_h^n)\|^2+\|(u^n-v^n)-(u_h^n-v_h^n)\|^2$,
multiplying by $k$ , taking the sum over $j$, then the following
inequality holds for all $\overline{u}_h^j,\, \overline{v}_h^j\in
S_h$
\begin{equation}\label{principal}
\begin{array}{ll}
\medskip
\displaystyle
e_n-e_0\leq \displaystyle Ck\sum_{j=0}^{n-1}\Bigg(\|\hat{u}_t^{j+1}-\delta \hat{u}^{j+1}\|^2+\|\hat{u}^{j+1}-\overline{u}_h^{j+1}\|^2+\|\hat{v}_t^{j+1}-\delta \hat{v}^{j+1}\|^2+\|\hat{v}^{j+1}-\overline{v}_h^{j+1}\|^2\\
\hspace{.5cm} \displaystyle +\|\hat{u}^{j+1}-\hat{u}_h^{j+1}\|^2+\|\hat{u}^{j+1}-\overline{u}_h^{j+1}\|^2+\|\hat{v}^{j+1}-\hat{v}_h^{j+1}\|^2+\|\hat{v}^{j+1}-\overline{v}_h^{j+1}\|^2\\
\hspace{.5cm} \displaystyle+\|\nabla(u^{j+1}-u_h^{j+1})\|^2+\|\nabla(\hat{u}^{j+1}-\overline{u}_h^{j+1})\|^2+\|\nabla(v^{j+1}-v_h^{j+1})\|^2+\|\nabla(\hat{v}^{j+1}-\overline{v}_h^{j+1})\|^2\\
\hspace{.5cm} \medskip +\|(u^{j+1}-v^{j+1})-(u_h^{j+1}-v_h^{j+1})\|^2+\|\nabla(\hat{u}^{j+1}-\delta u^{j+1})\|^2+\|\nabla(\hat{v}^{j+1}-\delta v^{j+1})\|^2\Bigg)\\
\hspace{.5cm} \medskip \displaystyle +\frac{C}{k}\sum_{j=0}^{n-2}\Bigg(\|(\hat{u}^{j+1}-\overline{u}_h^{j+1})-(\hat{u}^{j+2}-\overline{u}_h^{j+2})\|^2+\|(\hat{v}^{j+1}-\overline{v}_h^{j+1})-(\hat{v}^{j+2}-\overline{v}_h^{j+2})\|^2\Bigg)\\
\hspace{.5cm} \medskip \displaystyle
+C(\|\hat{u}^{n+1}-\overline{u}_h^{n+1}\|^2+\|\hat{u}_h^0-\hat{u}^0\|^2+\|\hat{u}^1-\overline{u}_h^1\|^2
+
\|\hat{v}^{n+1}-\overline{v}_h^{n+1}\|^2+\|\hat{v}_h^0-\hat{v}^0\|^2+\|\hat{v}^1-\overline{v}_h^1\|^2).
\end{array}
\end{equation}
Clearly \eqref{principal} can be rewritten as $$ e_n\leq
C\left(\sum_{j=0}^{n}ke_j +g_n\right),$$ Where
\begin{equation}\label{gn}
\begin{array}{rcl}
\medskip
\displaystyle
\begin{aligned}

g_n= & Ck\sum_{j=0}^{n-1}\Bigg(\|\hat{u}_t^{j+1}-\delta \hat{u}^{j+1}\|^2+\|\hat{u}^{j+1}-\overline{u}_h^{j+1}\|^2+\|\hat{v}_t^{j+1}-\delta \hat{v}^{j+1}\|^2+\|\hat{v}^{j+1}-\overline{v}_h^{j+1}\|^2\\

\medskip & +\|\nabla(\hat{u}^{j+1}-\overline{u}_h^{j+1})\|^2+\|\nabla(\hat{v}^{j+1}-\overline{v}_h^{j+1})\|^2+\|\nabla(\hat{u}^{j+1}-\delta u^{j+1})\|^2+\|\nabla(\hat{v}^{j+1}-\delta v^{j+1})\|^2\Bigg)\\
\medskip & +\frac{C}{k}\sum_{j=0}^{n-2}\Bigg(\|(\hat{u}^{j+1}-\overline{u}_h^{j+1})-(\hat{u}^{j+2}-\overline{u}_h^{j+2})\|^2+\|(\hat{v}^{j+1}-\overline{v}_h^{j+1})-(\hat{v}^{j+2}-\overline{v}_h^{j+2})\|^2\Bigg)\\
\medskip &+C(\|\hat{u}^{n+1}-\overline{u}_h^{n+1}\|^2+\|\hat{u}_h^0-\hat{u}^0\|^2+\|\hat{u}^1-\overline{u}_h^1\|^2)+C(\|\hat{v}^{n+1}-\overline{v}_h^{n+1}\|^2+\|\hat{v}_h^0-\hat{v}^0\|^2+\|\hat{v}^1-\overline{v}_h^1\|^2).
\end{aligned}
\end{array}
\end{equation}
The use of the discrete Gronwall inequality leads to the end of
the proof.
\end{proof}
\noindent As showed in \cite{HSS}, the following estimates holds
\begin{equation}\label{Han}
\frac{1}{k}\sum_{j=0}^{N-2}\|\hat{u}^{j+1}-R_h
\hat{u}^{j+1}-(\hat{u}^{j+2}-R_h \hat{u}^{j+2})\|^2\leq C
h^2\|\hat{u}_t\|_{L^2(0,T,H^1(\Omega))}^2.
\end{equation}
Combining \eqref{proj}, \eqref{Han} and \eqref{error-1} (with
$\overline{u}_h^j= R_h(u^j)$ and $\overline{v}_h^j= R_h(v^j)$)
give the following error estimate
\begin{corollary}
Under the assumption of Theorem \ref{thmapr} and the following
regularity assumptions $(u,v)\in
[W^{3,+\infty}(0,T,L^2(\Omega))\bigcup H^1(0,T,H^2(\Omega))]^2$,
the following a priori error estimate holds,
\begin{equation}\label{error-11}
\|\hat{u}^n-\hat{u}_h^n\|^2+\|\hat{v}^n-\hat{v}_h^n\|^2+\|\nabla(u^n-u_h^n)\|^2+\|\nabla(v^n-v_h^n)\|^2+\|(u^n-v^n)-(u_h^n-v_h^n)\|^2\leq
C(h^2+k^2),
\end{equation}
where $C$ is a positive constant independent of $h$ and $k$.

\end{corollary}

\section{Exponential decay of the discrete Energy}\label{Energydecay}
In this section, we introduce a discrete energy corresponding to problem \eqref{PDEd} and show the corresponding exponential decay.\\

\noindent Let $(u^{n+1}_{h}, v^{n+1}_{h})$ be the solution of the
problem \eqref{PDEd}. We introduce the following discrete energy:
\begin{equation}\label{energdisc}
E_h^{n+1}= \displaystyle \frac{1}{2} \int_\Omega \big(
|\frac{u_h^{n+1}-u_h^n}{k}|^2 + |\frac{v_h^{n+1}-v_h^n}{k}|^2 +
c^2 |\nabla u_h^{n+1} |^2 + c^2|\nabla v_h^{n+1} |^2 + \alpha
|u_h^{n+1} - v_h^{n+1}|^2 \big).
\end{equation}
\begin{lemma}
The positive energy $E_h^{n+1}$ decays with respect to $n$.
\end{lemma}
\begin{proof}
We take $\phi_h=u_h^{n+1}-u_h^n$ and $\psi_h = v_h^{n+1}-v_h^n$ in
System \eqref{PDEd} and use the relation $(b-a,b)=\displaystyle
\frac{1}{2}a^2 - \frac{1}{2} b^2 + \frac{1}{2} (b-a)^2$ to get
\[
\begin{array}{rcl}
\medskip
\displaystyle E_h^{n+1} - E_h^{n} &=& \displaystyle -\frac{1}{2} ||\frac{u_h^{n+1}- 2 u_h^n + u_h^{n-1}}{k} ||^2 -\frac{1}{2} ||\frac{v_h^{n+1}- 2 v_h^n + v_h^{n-1}}{k} ||^2 \\
\medskip
&& \displaystyle -\frac{c^2}{2} ||\nabla (u_h^{n+1}- u_h^n)||^2 -\frac{c^2}{2} ||\nabla (v_h^{n+1}- v_h^n)||^2 \\
\medskip
&& \displaystyle -\frac{\varepsilon_u}{k} || (u_h^{n+1}-
u_h^n)||^2 -\frac{\varepsilon_v}{k} || (v_h^{n+1}- v_h^n)||^2 \\
\medskip
&& + \displaystyle \frac{\alpha}{2} ||u_h^{n+1} - v_h^{n+1}||^2  - \frac{\alpha}{2} ||u_h^{n} - v_h^{n}||^2\\
&& - \alpha (u_h^{n+1} - v_h^{n+1}, u_h^{n+1}- u_h^n) -  \alpha
(v_h^{n+1} - u_h^{n+1}, v_h^{n+1}- v_h^n).
\end{array}
\]
The last two terms in the last equation can be written as follows:
\[
\begin{array}{l}
\medskip
\displaystyle - \alpha (u_h^{n+1} - v_h^{n+1}, u_h^{n+1}- u_h^n) -
\alpha (v_h^{n+1} - u_h^{n+1}, v_h^{n+1}- v_h^n) \\ \medskip
\quad = - \alpha (u_h^{n+1} - v_h^{n+1}, (u_h^{n+1} - v_h^{n+1})-(u_h^{n} - v_h^{n}))\\
\quad = - \displaystyle \frac{\alpha}{2} ||u_h^{n+1} -
v_h^{n+1}||^2 + \frac{\alpha}{2} ||u_h^{n} - v_h^{n}||^2 -
\frac{\alpha}{2} ||(u_h^{n+1} - v_h^{n+1})-(u_h^{n} -
v_h^{n})||^2.
\end{array}
\]
Using the last two equations leads to the following relation:
\begin{equation}\label{decayE}
\begin{array}{rcl}
\medskip
\displaystyle E_h^{n+1} - E_h^{n} &=& \displaystyle -\frac{1}{2} ||\frac{u_h^{n+1}- 2 u_h^n + u_h^{n-1}}{k} ||^2 -\frac{1}{2} ||\frac{v_h^{n+1}- 2 v_h^n + v_h^{n-1}}{k} ||^2 \\
\medskip
&& \displaystyle -\frac{c^2}{2} ||\nabla (u_h^{n+1}- u_h^n)||^2 -\frac{c^2}{2} ||\nabla (v_h^{n+1}- v_h^n)||^2 \\
\medskip
&& \displaystyle -\frac{\varepsilon_u}{k} || (u_h^{n+1}-
u_h^n)||^2 -\frac{\varepsilon_v}{k} || (v_h^{n+1}- v_h^n)||^2 \\
\medskip && \displaystyle - \frac{\alpha}{2} ||(u_h^{n+1} -
v_h^{n+1})-(u_h^{n} - v_h^{n})||^2.
\end{array}
\end{equation}
We deduce that $E_h^{n+1}-E_h^n \le 0$ and then the desired
result.
\end{proof}
\begin{lemma}
The energy $E_h^{n+1}$ decays exponentially with respect to the
$time$.
\end{lemma}
\begin{proof}
We proceed by steps:
\begin{enumerate}
\item We introduce the following discrete function
\[
L_h^{n+1} = N E_h^{n+1} + f_h^{n+1},
\]
where $f_h^{n+1} = \displaystyle \beta (\frac{u_h^{n+1}-u_h^n}{k},
u_h^{n+1}) + \beta (\frac{v_h^{n+1}-v_h^n}{k}, v_h^{n+1})$.
\item Let us now show that $L_h^{n+1}$ is equivalent to $E_h^{n+1}$: $C_m E_h^{n+1} \le L_h^{m+1} \le C_M E^{m+1}$ where  $C_m>0$ and $C_M>0$.\\
By using the Cauchy-Schwarz, the Young and Poincar\'{e}'s
inequalities (with constant $S_p$), the definitions of $f_h^{n+1}$
and $E_h^{n+1}$ give the following:
\[
\begin{array}{rcl}
\medskip
| L_h^{n+1} - N E_h^{n+1}| &=& |f_h^{n+1} |\\ \medskip &\le&
\displaystyle \frac{\beta}{2} ||\frac{u_h^{n+1}-u_h^n}{k}||^2 +
\frac{\beta S_p^2}{2} ||\nabla u_h^{n+1}||^2  \\ \medskip
&& \displaystyle +\frac{\beta}{2} ||\frac{v_h^{n+1}-v_h^n}{k}||^2 + \frac{\beta S_p^2}{2} ||\nabla v_h^{n+1}||^2\\
&\le & \displaystyle \beta \min(1,c^2 S_p^2) E_h^{n+1}.
\end{array}
\]
Finally, for $N$ sufficiently large, we obtain
\[
C_m E_h^{n+1} \le L_h^{n+1} \le C_M E_h^{n+1},
\]
where $C_m$ and $C_M$ are positive given constants.

%
%
%
%
\item Then, we show that $\displaystyle \frac{L_h^{n+1} -L_h^n}{k} \le - c_1 \; L_h^{n+1}$ with a positive constant $c_1$:\\
\noindent We have
\[
\begin{array}{rcl}
\medskip
\displaystyle \frac{f_h^{n+1}-f_h^n}{k} &=& \displaystyle \beta (\frac{u_h^{n+1}- 2 u_h^n + u_h^{n-1}}{k^2} , u_h^{n+1}) + \beta (\frac{u_h^n - u_h^{n-1}}{k^2} , u_h^{n+1}-u_h^n) \\
&& \displaystyle + \beta (\frac{v_h^{n+1}- 2 v_h^n +
v_h^{n-1}}{k^2} , v_h^{n+1}) + \beta (\frac{v_h^n -
v_h^{n-1}}{k^2} , v_h^{n+1}-v_h^n).
\end{array}
\]
By using scheme \eqref{PDEd} for $\phi_h = u_h^{n+1}$ and
$\psi_h=v_h^{n+1}$, we obtain the following:
\begin{equation}\label{fnp1n}
\begin{array}{rcl}
\medskip
\displaystyle \frac{f_h^{n+1}-f_h^n}{k} &=& \displaystyle -\beta
c^2 ||\nabla u_h^{n+1}||^2 -\beta \varepsilon_u
(\frac{u_h^{n+1}-u_h^n}{k}, u_h^{n+1}) \\ \medskip
&& \displaystyle -\beta c^2 ||\nabla v_h^{n+1}||^2 -\beta \varepsilon_v (\frac{v_h^{n+1}-v_h^n}{k}, v_h^{n+1}) - \beta \alpha ||v_h^{n+1} - u_h^{n+1}||^2 \\
&& + \displaystyle \beta (\frac{u_h^n - u_h^{n-1}}{k^2} ,
u_h^{n+1}-u_h^n) + \beta (\frac{v_h^n - v_h^{n-1}}{k^2} ,
v_h^{n+1}-v_h^n).
\end{array}
\end{equation}
By using the Cauchy-Schwarz inequality, the Young relation and the
poincarr\'e inequality, the second and fourth  terms of the right
hand side of the last relation can be bounded as follows:
\begin{equation}\label{secdforth}
\begin{array}{rcl}
\medskip
|T_1| &=& \displaystyle \Big| -\beta \varepsilon_u
(\frac{u_h^{n+1}-u_h^n}{k}, u_h^{n+1})  -\beta \varepsilon_v
(\frac{v_h^{n+1}-v_h^n}{k}, v_h^{n+1}) \Big| \\ \medskip
&\le& \displaystyle \frac{\beta \varepsilon_u}{2 \varepsilon_1} || \frac{u_h^{n+1}-u_h^n}{k}||^2 + \frac{\beta \varepsilon_v}{2 \varepsilon_2}  || \frac{v_h^{n+1}-v_h^n}{k} ||^2\\
&& + \displaystyle \frac{\beta \varepsilon_u \varepsilon_1
S_p^2}{2} || \nabla u_h^{n+1} ||^2 + \frac{\beta \varepsilon_v
\varepsilon_2 S_p^2}{2} || \nabla v_h^{n+1} ||^2.
\end{array}
\end{equation}

The last two terms of the relation \eqref{fnp1n} can be treated as
follows:
\begin{equation}\label{lastterm}
\begin{array}{rcl}
\medskip
|T_2| &=& \displaystyle \Big| - \beta (\frac{u_h^{n+1} - 2 u_h^n +
u_h^{n-1}}{k} , \frac{u_h^{n+1}-u_h^n}{k}) - \beta
(\frac{v_h^{n+1} - 2v_h^n + v_h^{n-1}}{k} ,
\frac{v_h^{n+1}-v_h^n}{k}) \Big| \\ \medskip
&& + \displaystyle \beta ||\frac{u_h^{n+1} -  u_h^n}{k}||^2 +
\beta ||\frac{v_h^{n+1} -  v_h^n}{k}||^2\\ \medskip
&\le&  \displaystyle \frac{1}{2} ||\frac{u_h^{n+1}- 2 u_h^n + u_h^{n-1}}{k} ||^2 +\frac{1}{2} ||\frac{v_h^{n+1}- 2 v_h^n + v_h^{n-1}}{k} ||^2\\
&& + \displaystyle (\beta + \frac{\beta^2}{2}) ||\frac{u_h^{n+1} -
u_h^n}{k}||^2 + (\beta + \frac{\beta^2}{2}) ||\frac{v_h^{n+1} -
v_h^n}{k}||^2.
\end{array}
\end{equation}
By considering the relations  \eqref{decayE} divided by $k$,
\eqref{fnp1n}, \eqref{secdforth} and \eqref{lastterm}, by using
the Cauchy-Schwarz and Young inequalities, and by neglecting some
negative terms, we obtain the following bound:
\begin{equation}\label{Lnp1n}
\begin{array}{rcl}
\medskip
\displaystyle \frac{L_h^{n+1} - L_h^n}{k} &\le& \displaystyle
-\frac{1}{2} (N-1) ||\frac{u_h^{n+1}- 2 u_h^n + u_h^{n-1}}{k} ||^2
-\frac{1}{2} (N-1)||\frac{v_h^{n+1}- 2 v_h^n + v_h^{n-1}}{k} ||^2
\\ \medskip
&&+  \displaystyle (-\varepsilon_u k N + \beta + \frac{\beta^2}{2}
+\frac{\beta \varepsilon_u}{2\varepsilon_1})|| \frac{u_h^{n+1}-
u_h^n}{k}||^2  \\ \medskip && \displaystyle + (-\varepsilon_v k N
+ \beta + \frac{\beta^2}{2} + \frac{\beta \varepsilon_v}{2
\varepsilon_2})|| \frac{v_h^{n+1}- v_h^n}{k}||^2 \displaystyle -
\beta \alpha ||v_h^{n+1} - u_h^{n+1}||^2 \\ \medskip
&& + \displaystyle (-\beta c^2 + \frac{\beta \varepsilon_u
\varepsilon_1 S_p^2}{2}) ||\nabla u_h^{n+1}||^2 + (-\beta c^2 +
\frac{\beta \varepsilon_v \varepsilon_2 S_p^2}{2}) ||\nabla
v_h^{n+1}||^2.
\end{array}
\end{equation}
For $\varepsilon_1 < \displaystyle \frac{2c^2}{\varepsilon_u
S_p^2}$, $\varepsilon_2 < \displaystyle \frac{2c^2}{\varepsilon_v
S_p^2}$ and $N$ sufficiently large, we get the following bound:
there exist a positive constant $c_1$ such that,
\[
\displaystyle \frac{L_h^{n+1} - L_h^n}{k} \le -c_1 L_h^{n+1}
\]
\item Exponential decay of the discrete energy: the last
inequality leads to the following:
\[
L_h^{n+1} \le \displaystyle \frac{1}{1+c_1 k} L_h^n
\]
which gives
\[
\begin{array}{rcl}
\medskip
L_h^{n+1} &\le& \displaystyle \frac{1}{(1+c_1 k)^n} L_h^1\\
&\le& \displaystyle e^{-n \ln (1+c_1 k)} L_h^1.
\end{array}
\]
We deduce finally that
\[
E_h^{n+1} \le \displaystyle \frac{C_M}{C_m} e^{-n \ln (1+c_1k)}
E_h^1.
\]
Hence we get the following desired result:
\[
E_h^{n+1} \le \displaystyle C e^{- \gamma t_n}
\]
where $\gamma = \displaystyle \frac{\ln (1+c_1k)}{k}$ and
$C=\displaystyle \frac{C_M}{C_m}  E_h^1$.
\end{enumerate}
\end{proof}

\end{document}